\documentclass[12pt,a4paper,reqno]{amsart}
\usepackage{amsmath}
\usepackage{amsfonts}
\usepackage{amssymb,amsthm,amsfonts,amsthm,latexsym,enumerate,url,cases}
\numberwithin{equation}{section}
\usepackage{mathrsfs}
\usepackage{hyperref}
\hypersetup{colorlinks=true,citecolor=blue,linkcolor=blue,urlcolor=blue}
\def\pmod #1{\ ({\rm{mod}}\ #1)}

\theoremstyle{plain}
\newtheorem{theorem}{Theorem}
\newtheorem{lemma}{Lemma}
\newtheorem{problem}{Problem}
\newtheorem{corollary}{Corollary}
\newtheorem{proposition}{Proposition}

\theoremstyle{definition}

\newtheorem{remark}{Remark}

\usepackage{etoolbox}
\makeatletter
\patchcmd{\@settitle}{\uppercasenonmath\@title}{}{}{}
\patchcmd{\@setauthors}{\MakeUppercase}{}{}{}
\patchcmd{\section}{\scshape}{}{}{}
\makeatother

\begin{document}

\title
[{On a problem of Romanoff type}]
{On a problem of Romanoff type}

\author
[Yuchen Ding]
{Yuchen Ding}

\address{(Yuchen Ding) School of Mathematical Science,  Yangzhou University, Yangzhou 225002, People's Republic of China}
\email{ycding@yzu.edu.cn}

\keywords{Romanoff theorem, Polignac conjecture, Prime set}
\subjclass[2010]{11P32, 11A41, 11B13.}

\begin{abstract} Let $\mathcal{P}$ and $\mathbb{N}$ be the sets of all primes and natural numbers, respectively. In this article, it is proved that there is a positive lower density of the natural numbers which can be represented by the form $$p+2^{m_1^2}+2^{m_2^2}~~(p\in \mathcal{P},m_1,m_2\in \mathbb{N}).$$
This solves a problem of Chen and Yang in 2014.
\end{abstract}
\maketitle

\section{Introduction}
In 1849, de Polignac \cite{de1} conjectured that every odd number greater than 1 is the sum of a prime and a power of 2. However, de Polignac \cite{de2} himself soon recognized that 127 and 959 are two counterexamples. Actually, he pointed out that Euler had already found the same counterexamples many years ago. Later, de Polignac's report aroused great interest to Romanoff. A remarkable theorem of Romanoff \cite{Ro} states that the odd numbers which can be represented by the form $$p+2^m~(p\in\mathcal{P},m\in \mathbb{N})$$ have a positive lower density. In the opposite direction, van der Corput \cite{va} proved that the odd numbers greater than 1 which can not be represented by the form $p+2^m$ with $p\in\mathcal{P}$ and $m\in \mathbb{N}$ still possess a positive lower density of all the natural numbers. In 1950, answering a question of Romanoff, Erd\H{o}s \cite{Er} constructed a specific arithmetic progression, none of which can be written as the sum of a prime and a power of 2. This is surprising since it is easy to see that van der Corput's result follows immediately from the one of Erd\H{o}s.

The quantitative version of Romanoff's theorem was first considered by Chen and Sun \cite{CS}, who showed that the lower density is larger than 0.0868. There are a few improvements of this lower density, see \cite{ES,Lv,Pi}. In the opposite direction as pointed by Chen and Sun, Erd\H{o}s' construction  \cite{Er} illustrated that the lower density is no larger than 0.49999991. Suggested by Bombieri, Pintz \cite{Pi} had conjectured that there is a constant $c_0$(=0.434...) such that
$$\lim_{x\rightarrow\infty}\frac{\#\{n:n\leqslant x,~ n=p+2^m, ~m\in\mathbb{N},~p\in\mathcal{P}\}}{x}=c_0.$$
Recently, Del Corso et al. \cite{DDD} thought that $0.437...$ is another probable candidate of the constant $c_0$.
There are a huge number of variants of Romanoff's theorem and the readers may refer to \cite{Ch1,Ch2,Ch3,Ch4,Ch5,CS,Cr,Cr2,Els,FFK,Pa1,Pa2,Pa3,SF,Ta,WS,Yu}.

Let $x$ be a large number. Define $$V=\left\{n:n=p+2^{m_1^2}+2^{m_2^2},p\in \mathcal{P},m_1,m_2\in \mathbb{N}\right\}$$
and $V(x)=|V\cap[1,x]|$.
In 2014, Chen and Yang \cite{Ch6} proved that
$$V(x)\gg\frac{x}{\log\log x}$$
and posed the following problem for further research.
\begin{problem}[Chen--Yang]\label{Pro1}Does there exist a positive integer $k$ such that the set of positive integers which can be represented as $p+\sum_{i=1}^{k}2^{m_i^2}$ with $p\in\mathcal{P}$ and $m_i\in \mathbb{N}$ has a positive lower density? If such $k$ exists, what is the minimal value of such $k$.
\end{problem}
Certainly, the problem of Chen and Yang is still a new variant of the Romanoff theorem. In this article, we shall prove that $k=2$ is admissible, i.e., $V(x)\gg x$. It is clear that $k=1$ does not satisfy the requirement of the problem since
$$\#\{p+2^{m^2}:p\in \mathcal{P},m\in \mathbb{N}\}\ll\frac{x}{\log x}\cdot\sqrt{\log x}=\frac{x}{\sqrt{\log x}}.$$
Thereby the minimal value of $k$ is $2$. Thus, we give a complete answer to the above problem. Now, we restate it as the following theorem.
\begin{theorem}\label{thm1} Let $V$ be defined as above. Then we have $V(x)\gg x.$
\end{theorem}
\begin{remark} Actually, one can prove a slightly stronger result using the same arguments.
$$V_g(x)=\#\left\{n:n\leqslant x,~n=p+g^{m_1^2}+g^{m_2^2},p\in \mathcal{P},m_1,m_2\in \mathbb{N}\right\}\gg_gx,$$
where $g$ is a fixed positive integer larger than 1. We do not proceed this general case as the consideration of powers of $2$ in Romanoff type problems is a common practice.
\end{remark}

From the Theorem \ref{thm1}, it follows clearly that we have the following corollary by the Shnirel'man density (see for example \cite{Na}).

\begin{corollary}\label{coro1} There exists an integer $s$ such that every sufficiently large integer can be represented by the form
$$p_1+\cdot\cdot\cdot+p_s+2^{m_1^2}+\cdot\cdot\cdot+2^{m_{2s}^2},$$
where $p_i\in \mathcal{P}$ for $1\leqslant i\leqslant s$ and $m_j\in \mathbb{N}$ for $1\leqslant j\leqslant 2s$.
\end{corollary}

\section{Some Auxiliary Results}
Let $a,m,r,k,d,\ell$ and $n$ be integers and $y$ be a positive number. The order of $a$ modulo $d$ is the smallest positive integer $\ell$ such that
$$a^\ell\equiv 1\pmod{d}.$$
The aim of this section is to establish two  propositions which shall be used in the proof of our theorem.
\begin{proposition}\label{pro1} The number of the solutions to the congruence equation
$$z^2\equiv a\pmod{m}$$
is at most $4\sqrt{m}$.
\end{proposition}
The proof of Proposition \ref{pro1} follows directly from the following two lemmas. 

\begin{lemma}\label{lem1}\cite[Corollary 0.3.3]{Le} Let $f(z)$ be a quadratic polynomial with integer coefficients. If $m$ is a positive integer, let $n_m(f)$ be the number of solutions of $f(z)\equiv 0\pmod{m}$ in $\mathbb{Z}/m\mathbb{Z}$. Then
$$n_m(f)=\prod_{p}n_{p^e}(f),$$
where the product is taken over all primes, and where $e=e_p(m)$ is the exponent of $p$ in $m$, that is, the largest integer so that $p^e$ divides $m$. 
\end{lemma}

Let $\Delta$ be an integer congruent to $0$ or $1$ modulo $4$, define
\begin{equation}\label{Eq2.1}
\left(\frac{\Delta}{2}\right)=
\begin{cases}
 1,& \text{if~} \Delta\equiv 1\pmod{8},\\
 -1,& \text{if~} \Delta\equiv 5\pmod{8},\\
 0, & \text{if~} \Delta\equiv 0\pmod{4}.\\
\end{cases}
\end{equation}
\begin{lemma}\label{lem2}\cite[Theorem 0.3.4]{Le} Let $f(z)=az^2+bz+c$, with discriminant $\Delta=b^2-4ac$, and let $p$ be a prime number not dividing $a$. If $\Delta\neq0$, let $\ell\geqslant0$ be the largest integer for which $\Delta=p^{2\ell}\Delta_0$ with $\Delta_0\equiv 0$ or $1\pmod{4}$. Let $e$ be a nonnegative integer. If $e>2\ell$, then the following statements are true.\\
(1) If $\left(\frac{\Delta_0}{p}\right)=1$, then $n_{p^e}(f)=2p^{\ell}$.\\
(2) If $\left(\frac{\Delta_0}{p}\right)=-1$, then $n_{p^e}(f)=0$.\\
(3) If $\left(\frac{\Delta_0}{p}\right)=0$, then 
$n_{p^e}(f)=
\begin{cases}
p^{\ell},& \text{if~} e=2\ell+1,\\
0,& \text{if~} e>2\ell+1.
\end{cases}$\\
On the other hand, if $e<2\ell$ is written as $e=2k+r$ with $r=0$ or $1$, then $n_{p^e}(f)=p^k$. This equation also holds for all $e\geqslant0$ if $\Delta=0$.
\end{lemma}

Now we turn back to the proof of Proposition \ref{pro1}.

\begin{proof}[Proof of Proposition \ref{pro1}] Let
$$m=p_1^{e_1}\cdot\cdot\cdot p_t^{e_t}$$
be the standard factorization of $m$. Note that $2<\sqrt{p}$ for all $p>3$, which means that
$n_{p^e}(f)<\sqrt{p^e}$ for $p>3$ and $n_{p^e}(f)<2\sqrt{p^e}$ for $p=2,3$ from Lemma \ref{lem2}.
Taking $f(z)=z^2+a$ in Lemma \ref{lem1} and Lemma \ref{lem2} gives
$$n_m(f)=\prod_{i=1}^{t}n_{p_i^{e_i}}(f)<2^2\prod_{i=1}^t\sqrt{p_i^{e_i}}=4\sqrt{m}.$$
\end{proof}

The second proposition to be established is the following one.
\begin{proposition}\label{pro2} Let $m$ and $a$ be two integers and $2\leqslant y\leqslant m$ be a positive number. Denoting by $N(y,m;a)$ the number of the solutions to the congruence equation
$$z^2\equiv a\pmod{m}$$
with $1\leqslant z\leqslant y$, then we have
$$N(y,m;a)\ll y^{2/3}.$$
\end{proposition}
\begin{proof}The proof of the proposition shall be divided into two cases. Suppose first that $m<y^{4/3}$, then we have
$$N(y,m;a)\leqslant4\sqrt{m}\ll\sqrt{m}<y^{2/3}$$
by the Proposition \ref{pro1}.
We now assume $m\geqslant y^{4/3}$. Let $t$ be the least integer such that $y< \sqrt{tm}$, which implies that
$$t=\left\lfloor\frac{y}{\sqrt{m}}\right\rfloor+1\leqslant\frac{y}{\sqrt{m}}+1.$$
Consider the following $t$ short intervals
$$[1,\sqrt{m}),~[\sqrt{m},\sqrt{2m}),~\cdot\cdot\cdot,~[\sqrt{(t-1)m},\sqrt{tm}).$$
It can be seen that there exists at most one integer $z$ between each interval satisfying the congruence equation
$$z^2\equiv a\pmod{m}.$$
Thus we have
$$N(y,m;a)\leqslant t\leqslant\frac{y^2}{m}+1\ll y^{2/3}.$$
\end{proof}
\section{Proof of The Theorem}
\begin{proof}[Proof of Theorem \ref{thm1}]
Let
$$r(n)=\#\left\{(p,m_1,m_2):n=p+2^{m_1^2}+2^{m_2^2},p\in \mathcal{P},m_1,m_2\in \mathbb{N}\right\}$$
be the corresponding representation function. In view of the Cauchy--Schwarz inequality, we have
\begin{equation}\label{e1}
\left(\sum_{n\leqslant x}r(n)\right)^2\leqslant V(x)\left(\sum_{n\leqslant x}r^2(n)\right).
\end{equation}
It is trivial that
\begin{equation}\label{e2}
\sum_{n\leqslant x}r(n)=\sum_{p+2^{m_1^2}+2^{m_2^2}\leqslant x}1\geqslant\pi\left(\frac{x}{3}\right)\cdot\left\lfloor\sqrt{\frac{\log(x/3)}{\log 2}}\right\rfloor\cdot\left\lfloor\sqrt{\frac{\log(x/3)}{\log 2}}\right\rfloor\gg x
\end{equation}
from the prime number theorem, where $\pi(x)$ is the number of primes below $x$. Hence, it remains to prove $\sum_{n\leqslant x}r^2(n)\ll x$. By the definition of $r(n)$, it can be seen that
\begin{align}\label{e3}
\sum_{n\leqslant x}r^2(n)=\sum_{n\leqslant x}\left(\sum_{n=p+2^{m_1^2}+2^{m_2^2}}1\right)^2
=\sum_{p+2^{m_1^2}+2^{m_2^2}=q+2^{k_1^2}+2^{k_2^2}\leqslant x}1,
\end{align}
where $p,q$ are primes and $m,k$ are natural numbers to the rest of the proof.
That is to say, $\sum_{n\leqslant x}r^2(n)$ does not exceed the number of the solutions to the following constraints
\begin{equation}\label{c1}
p+2^{m_1^2}+2^{m_2^2}=q+2^{k_1^2}+2^{k_2^2},~p,q\leqslant x,~m_i,k_i\leqslant\sqrt{\log x/\log2},~i=1,2.
\end{equation}
From now on, let
\begin{equation}\label{e4}
h=2^{m_1^2}+2^{m_2^2}-2^{k_1^2}-2^{k_2^2}.
\end{equation}
First we consider the case $h=0$. Without loss of generality, we can suppose that $m_1\geqslant k_1$. In this case we have
\begin{equation}\label{e5}
2^{k_1^2}\left(2^{m_1^2-k_1^2}-1\right)=2^{m_2^2}\left(2^{k_2^2-m_2^2}-1\right).
\end{equation}
Note that both $2^{m_1^2-k_1^2}-1$ and $2^{k_2^2-m_2^2}-1$ are odd numbers if $m_1\neq k_1$, which means $k_1=m_2$ and $k_2=m_1$ in this situation. Thus we conclude that $h=0$ if and only if $\{m_1,m_2\}=\{k_1,k_2\}$.
For $h\neq0$, let $\pi_2(x,h)$ be the number of the prime pairs $p$ and $q$ with $q-p=h$ below $x$. It is well known (see for example \cite[Theorem 7.3]{Na}) that
\begin{align}\label{c2}
\pi_2(x,h)\ll\frac{x}{\log^2x}\prod_{p|h}\left(1+\frac{1}{p}\right).
\end{align}
Therefore from the analysis above and equation (\ref{c2}), it follows that
\begin{align}\label{e6}
\sum_{n\leqslant x}r^2(n)\ll\sum_{\substack{m_1,m_2\leqslant\sqrt{\frac{\log x}{\log2}}\\k_1=m_1,k_2=m_2}}\pi(x)+\sum_{\substack{m_i,k_j\leqslant\sqrt{\frac{\log x}{\log2}}\\i,j=1,2\\h\neq0}}\frac{x}{\log^2x}\prod_{p|2^{m_1^2}+2^{m_2^2}-2^{k_1^2}-2^{k_2^2}}\left(1+\frac{1}{p}\right).
\end{align}
Still by the prime number theorem, we have
\begin{equation}\label{e7}
\sum_{\substack{m_1,m_2\leqslant\sqrt{\frac{\log x}{\log2}}\\k_1=m_1,k_2=m_2}}\pi(x)\ll x.
\end{equation}
So it remains to prove
\begin{equation}\label{e8}
\sum_{\substack{m_i,k_j\leqslant\sqrt{\frac{\log x}{\log2}}\\i,j=1,2\\h\neq0}}\prod_{p|2^{m_1^2}+2^{m_2^2}-2^{k_1^2}-2^{k_2^2}}\left(1+\frac{1}{p}\right)\ll\log^2x.
\end{equation}
Since the factor $1+\frac{1}{p}$ for $p=2$ is bounded by $\frac{3}{2}$ which can be neglected, we need only to consider odd primes $p$ in the equation (\ref{e8}). Following an idea of Elsholtz et al. \cite{Els}, we split the product into the following two parts:
\begin{equation}\label{twpa}
\prod_{\substack{p|2^{m_1^2}+2^{m_2^2}-2^{k_1^2}-2^{k_2^2}\\2<p<\log x}}\left(1+\frac{1}{p}\right)~~~~\text{and}~~~~\prod_{\substack{p|2^{m_1^2}+2^{m_2^2}-2^{k_1^2}-2^{k_2^2}
\\p\geqslant\log x}}\left(1+\frac{1}{p}\right).
\end{equation}
Note that $2^{m_1^2}+2^{m_2^2}-2^{k_1^2}-2^{k_2^2}<2x$, we know that the number of the primes satisfying the constraints $p\geqslant\log x$ and $p|2^{m_1^2}+2^{m_2^2}-2^{k_1^2}-2^{k_2^2}$ is no larger than $\frac{\log 2x}{\log 2}$. Hence the second product can be bounded as
\begin{equation}\label{eq1}
\prod_{\substack{p|2^{m_1^2}+2^{m_2^2}-2^{k_1^2}-2^{k_2^2}
\\p\geqslant\log x}}\left(1+\frac{1}{p}\right)\ll\left(1+{\frac{1}{\log x}}\right)^{\frac{\log 2x}{\log 2}}\ll1.
\end{equation}
Thus to prove the equation (\ref{e8}), it suffices to prove
\begin{equation}\label{eq2}
\sum_{\substack{m_i,k_j\leqslant\sqrt{\frac{\log x}{\log2}}\\i,j=1,2\\h\neq0}}\prod_{\substack{p|2^{m_1^2}+2^{m_2^2}-2^{k_1^2}-2^{k_2^2}\\2<p<\log x}}\left(1+\frac{1}{p}\right)
\ll\log^2x.
\end{equation}
It is not difficult to see that
\begin{equation}\label{e9}
\sum_{\substack{m_i,k_j\leqslant\sqrt{\frac{\log x}{\log2}}\\i,j=1,2\\h\neq0}}\prod_{\substack{p|2^{m_1^2}+2^{m_2^2}-2^{k_1^2}-2^{k_2^2}\\2<p<\log x}}\left(1+\frac{1}{p}\right)=
\sum_{\substack{d<2x\\2\nmid d\\P^+(d)<\log x}}\frac{\mu^2(d)}{d}\sum_{\substack{m_i,k_i\leqslant\sqrt{\frac{\log x}{\log2}}\\i,j=1,2,~h\neq0\\d|2^{m_1^2}+2^{m_2^2}-2^{k_1^2}-2^{k_2^2}}}1,
\end{equation}
where $\mu(d)$ is the M\"{o}bius function and $P^+(d)$ is the largest prime factor of $d$.
For any fixed integers $d$, $m_1,~m_2$ and $k_1$ with $2\nmid d$ and $d|2^{m_1^2}+2^{m_2^2}-2^{k_1^2}-2^{k_2^2}$, there is some $l_{m_1,m_2,k_1}$ such that
\begin{align}\label{e10}
2^{k_2^2}\equiv2^{m_1^2}+2^{m_2^2}-2^{k_1^2}\equiv l_{m_1,m_2,k_1}\pmod{d}.
\end{align}
Let $e_2(d)$ stand for the order of $2$ modulo $d$. Then from the congruent condition of the equation (\ref{e10}), we know that there exists some $\widetilde{l}_{m_1,m_2,k_1}$ such that
\begin{equation}\label{e11}
k_2^2\equiv\widetilde{l}_{m_1,m_2,k_1}\pmod{e_2(d)},
\end{equation}
providing that there do exist positive integers $k_2$ satisfying the equation (\ref{e10}).
For $e_2(d)\leqslant \sqrt{\frac{\log x}{\log2}}$, the number of the solutions to the congruence equation (\ref{e11}) is at most $4\sqrt{e_2(d)}$ from Proposition \ref{pro1}. Thus there are at most
\begin{equation}\label{e12}
\left(\left\lfloor\frac{\sqrt{\frac{\log x}{\log2}}}{e_2(d)}\right\rfloor+1\right)4\sqrt{e_2(d)}
\end{equation}
choices of $k_2$ below $\sqrt{\frac{\log x}{\log2}}$ if $d,m_1,~m_2$ and $k_1$ are fixed for $e_2(d)\leqslant \sqrt{\frac{\log x}{\log2}}$. For $e_2(d)> \sqrt{\frac{\log x}{\log2}}$, the number of the solutions to the congruence equation (\ref{e11}) below $\sqrt{\frac{\log x}{\log2}}$ is at most $O\left(\frac{\sqrt{\log x}}{\log\log x}\right)$ from the Proposition \ref{pro2}.
Taking these bounds of the solutions to the congruent condition into the equation (\ref{e9}), we obtain that

\begin{small}
\begin{align}\label{e13}
\sum_{\substack{d<2x\\2\nmid d\\P^+(d)<\log x}}\!\!\!\frac{\mu^2(d)}{d}\!\!\!\!\!\!\sum_{\substack{m_i,k_j\leqslant\sqrt{\frac{\log x}{\log2}}\\i,j=1,2,~h\neq0\\d|2^{m_1^2}+2^{m_2^2}-2^{k_1^2}-2^{k_2^2}}}1
\ll&\sum_{\substack{d<2x\\2\nmid d\\P^+(d)<\log x\\e_2(d)\leqslant\sqrt{\frac{\log x}{\log2}}}}\!\!\!\!\!\frac{\mu^2(d)}{d}\left(\sqrt{\frac{\log x}{\log2}}\right)^3\left(\left\lfloor\frac{\sqrt{\frac{\log x}{\log2}}}{e_2(d)}\right\rfloor+1\right)4\sqrt{e_2(d)}\nonumber\\
&~~~+\sum_{\substack{d<2x\\2\nmid d\\P^+(d)<\log x\\e_2(d)>\sqrt{\frac{\log x}{\log2}}}}\!\!\!\frac{\mu^2(d)}{d}\left(\sqrt{\frac{\log x}{\log2}}\right)^3\frac{\sqrt{\log x}}{\log\log x}\nonumber\\
\ll& (\log x)^2\!\!\!\!\!\!\sum_{\substack{d<2x\\2\nmid d\\P^+(d)<\log x\\e_2(d)\leqslant\sqrt{\frac{\log x}{\log2}}}}\!\!\!\!\frac{\mu^2(d)}{d\sqrt{e_2(d)}}+\frac{(\log x)^2}{\log\log x}\!\!\!\!\!\!\sum_{\substack{d<2x\\2\nmid d\\P^+(d)<\log x\\e_2(d)>\sqrt{\frac{\log x}{\log2}}}}\!\!\!\frac{\mu^2(d)}{d},\nonumber\\
\ll& (\log x)^2S_1(x)+\frac{(\log x)^2}{\log\log x}S_2(x),~\text{say}.
\end{align}
\end{small}
By a result of Erd\H{o}s and Tur\'{a}n \cite{ET}, we have
\begin{equation}\label{e17}
\sum_{2\nmid d}\frac{1}{d\sqrt{e_2(d)}}<\infty.
\end{equation}
Therefore the sum $S_1(x)$ can be bounded as
\begin{equation}\label{e18}
S_1(x)=\sum_{\substack{d<2x\\2\nmid d\\P^+(d)<\log x\\e_2(d)\leqslant\sqrt{\frac{\log x}{\log2}}}}\frac{1}{d\sqrt{e_2(d)}}\leqslant\sum_{\substack{2\nmid d}}\frac{1}{d\sqrt{e_2(d)}}<\infty.
\end{equation}
To deal with the sum $S_2(x)$, we exchange the summations back to the product.
\begin{align}\label{e19}
S_2(x)=\sum_{\substack{d<2x\\2\nmid d\\P^+(d)<\log x\\e_2(d)>\sqrt{\frac{\log x}{\log2}}}}\frac{\mu^2(d)}{d}<\sum_{\substack{d<2x\\2\nmid d\\P^+(d)<\log x}}\frac{\mu^2(d)}{d}
\leqslant\prod_{2<p<\log x}\left(1+\frac{1}{p}\right).
\end{align}
By the logarithmic inequality
\begin{equation}\label{e21}
\log\left(1+\frac{1}{p}\right)<\frac{1}{p}
\end{equation}
and the Mertens formula
\begin{equation}\label{e20}
\sum_{p<\log x}\frac{1}{p}=\log\log\log x+B+O\left(\frac{1}{\log\log x}\right),
\end{equation}
we obtain that
\begin{align}\label{e25}
\prod_{2<p<\log x}\left(1+\frac{1}{p}\right)\leqslant \exp\left(\log\log\log x+B+O\left(\frac{1}{\log\log x}\right)\right)\ll \log\log x.
\end{align}
The proof of the theorem is complete by combining the equations (\ref{e8}), (\ref{eq2}), (\ref{e9}), (\ref{e13}), (\ref{e18}), (\ref{e19}) and (\ref{e25}).
\end{proof}

\section*{Acknowledgments}
The author would like to thank the referee for his helpful comments which improved the quality of the paper greatly. The author also would like to thank doctor Guilin Li and doctor Li--Yuan Wang for their careful reading of the manuscript. 

The author was supported by the Natural Science Foundation of Jiangsu Province of China, Grant No. BK20210784. He was also supported by the foundations of the projects ``Jiangsu Provincial Double--Innovation Doctor Program'', Grant No. JSSCBS20211023 and ``Golden  Phoenix of the Green City--Yang Zhou'' to excellent Ph.D., Grant No. YZLYJF2020PHD051.

\end{document}